\documentclass[psamsfonts]{amsart}

\usepackage{amssymb}
\usepackage{amsmath}
\usepackage{amsthm}
\usepackage[pdftex]{hyperref}

\hypersetup{
    colorlinks=true,
    linkcolor=blue,
    filecolor=blue,      
    urlcolor=blue,
    citecolor=blue
}

\allowdisplaybreaks

 \newtheorem{thm}{Theorem}[section]
 \newtheorem{prop}[thm]{Proposition}
 \newtheorem{lem}[thm]{Lemma}
 \newtheorem{cor}[thm]{Corollary}
\theoremstyle{definition}
 
 \newtheorem{dfn}[thm]{Definition}
\newtheorem{rem}[thm]{Remark}
\newtheorem{conv}[thm]{Convention}
 \numberwithin{equation}{section}
\newtheorem{mainthm}{ Theorem}

\renewcommand{\a}{\alpha}

\renewcommand{\d}{\delta}

\newcommand{\e}{\varepsilon}
\newcommand{\g}{\gamma}
\newcommand{\G}{\Gamma}
\renewcommand{\l}{\lambda}
\renewcommand{\L}{\Lambda}
\newcommand{\n}{\nabla}

\newcommand{\s}{\sigma}

\renewcommand{\t}{\tau}

\renewcommand{\leq}{\leqslant}
\renewcommand{\geq}{\geqslant}
\renewcommand{\setminus}{\smallsetminus}
\title[Curvature Free Rigidity for Higher Rank Three-manifolds]{CURVATURE FREE RIGIDITY FOR HIGHER RANK THREE-MANIFOLDS}

\author[Lin]{\bfseries Samuel Lin}

\address{
Department of Mathematics \\ 
Michigan State university   \\ 
East lansing, MI, 48823\\
USA}
\email{linsamue@msu.edu}

\dedicatory{ }

\begin{document}

\vspace{18mm} \setcounter{page}{1} \thispagestyle{empty}

\begin{abstract}
We prove two rigidity results for complete Riemannian three-manifolds of higher rank. Complete three-manifolds have higher spherical rank if an only if they are spherical space forms. Complete finite volume three-manifolds have higher hyperbolic rank if and only if they are hyperbolic space forms. 

\end{abstract}

\maketitle
\section{Introduction}
This paper proves rigidity results for three-manifolds in terms of geometric notions of rank first introduced in \cite{MR819559}, \cite{MR908215}, \cite{MR1114460} and \cite{MR2137949}.
Fix $\e=-1, 0, \text{or } 1$. A complete Riemannian manifold is said to have higher hyperbolic, Euclidean, or spherical rank if every geodesic admits a normal
parallel field making constant sectional curvature $\e$ with the geodesic.

Historically, these notions of rank have been studied in conjunction with sectional curvature bounds.
 Ballmann \cite{MR819559}, 
and independently Burns and Spatzier \cite{MR908215}, initiated the study of higher rank manifolds. They proved that
a finite volume, locally irreducible, manifold of higher Euclidean rank is locally symmetric if $-a^2<\mathrm{sec}\leq0$ for some constant $a$. This result was generalized in \cite{MR1079642} by Eberlein and Heber. Hamenst{\"a}dt generalized the notion of rank into the hyperbolic setting, proving that compact manifolds of higher hyperbolic rank are locally symmetric if $\mathrm{sec}\leq-1$ \cite{MR1114460}. This curvature condition was relaxed by Constantine to $\mathrm{sec}\leq0$ under additional dimensional or pinching assumptions \cite{MR2449143}. For manifolds of higher spherical rank, Shankar, Spatzier, and Wilking showed that
they are locally symmetric if $\mathrm{sec}\leq1$ \cite{MR2137949}. Partial spherical rank rigidity results were obtained by Schmidt, Shankar, and Spatzier when $\mathrm{sec}\geq1$ \cite{MR3493370}.

In contrast, there are fewer rank rigidity results without a priori sectional curvature bounds. Existing results include the works of 
 Molina-Olmos \cite{MR1472889,MR1860505}, Watkins \cite{MR3054631}, and Bettiol-Schmidt \cite{Bettiol:2014aa}, where Euclidean rank-rigidity 
results are proved after replacing curvature assumptions with assumptions on having many flats, having no focal points, and dimension, respectively.

This paper presents the first hyperbolic and spherical rank rigidity results after replacing a priori sectional curvature bounds with a dimensional 
assumption.

 \begin {mainthm} \label{thm_2}
A finite volume complete Riemannian three-manifold $M$ has higher hyperbolic rank if and only if $M$ is a hyperbolic
space form.
\end{mainthm}

 \begin{mainthm} \label {thm_1}
A complete Riemannian three-manifold $M$ has higher spherical rank if and only if $M$ is a spherical space form. 
\end{mainthm}

The finite volume assumption in Theorem \ref{thm_2} is necessary. In each dimension $d\geq3$ and for each $\d>0$, there exist complete higher hyperbolic rank metrics on
$\mathbb{R}^d$ having non-constant pinched sectional curvatures $-1-\d \leq \mathrm{sec}\leq -1$ or $-1 \leq \mathrm{sec}\leq -1+\d$, as will be proved in \cite{P01}.
It should be noted that a non-symmetric manifold with infinite volume and higher hyperbolic rank has already been constructed by Connell in \cite{MR1934699}.
His example is five dimensionl and homogeneous. Our final result shows that no such examples exist in dimension three.

\begin{mainthm} \label{thm_3}
A homogeneous three-manifold of higher hyperbolic rank is a hyperbolic space form.
\end{mainthm}

Our strategy is to analyze the local structure of three-manifolds of higher rank in terms of Ricci diagonalizing orthonormal frames. 
This is roughly done as follows. At non-isotropic points, the Ricci tensor either has two or three distinct eigenvalues. 
The set of points where the Ricci tensor has three distinct eigenvalues is called the generic set. The majority of work 
in this paper is to show that on the generic set, the Christoffel symbols of the Ricci diagonalizing frame satisfy an overdetermined system of differential equations.
From the system of equations, we deduce that the generic set must be empty for manifolds of higher spherical rank or finite volume manifolds of higher
hyperbolic rank. Hence the rigidity problem for a three-manifold $M$ of higher hyperbolic or spherical rank is reduced to the case when $M$ has \textit{extremal curvatures}, that is 
$\mathrm{sec}_p\geq \e$ or $\mathrm{sec}_p\leq \e$ for each $p\in M$. 

Three-manifolds with $\mathrm{cvc}(\e)$, a pointwise notion of having higher rank, and with extremal curvature $\e$ are studied by Schmidt and Wolfson in \cite{MR3298721}.
Their structural results are strengthened in this paper by the higher rank assumption, leading to the desired rigidity results.

For three-manifolds of higher Euclidean rank, the system of equations on the generic set still hold.
However, it appears to the author that these equations alone are not strong enough to imply a rigidity result.

In Section 2, we introduce notations and preliminary results. Local structures for nonisotropic sets are studied
in Section 3 and 4 in terms of Ricci diagonalizing orthonormal frames. This allows us to derive differential equations along a certain family of geodesics  
 on an open subset of the nonisotropic set in Section 5 and 6.  In Section 7, we apply the equations to prove 
the main theorems.

\section {Preliminaries}
This section introduces notations and preliminary results. Most of the results in this
section can be found in \cite{MR3298721}.
Throughout, $M$ denotes a complete Riemannian three-manifold with Levi-Civita connection $\n$ and curvature tensor $R$.
For each $p\in M$, let $T_pM$ and $U_pM$ denote the tangent space and unit tangent space at $p$ respectively. 

As $M$ is three dimensional, an orthonormal frame $\{e_1, e_2, e_3\}$ of $T_pM$ 
diagonalizes the Ricci curvature tensor if and only if
\begin{equation}\label{Ricci_diag}
R(e_1, e_2,e_3,e_1)=R(e_1, e_2,e_2,e_3)=R(e_1, e_3,e_3,e_2)=0.
\end{equation}
Let $\l_{ij}=\mathrm{sec}(e_i, e_j)$. Up to permuting indices,
\begin{equation} \label{per}
\l_{13}\leq\l_{12}\leq\l_{23}.
\end{equation}
All sectional curvatures are computable in terms of the $\l_{ij}$ as described in the next lemma proved in \cite[Lemma 2.2]{MR3298721}.
\begin{lem}  \label{sectional}
Let $\s$ be a 2-plane with unit normal vector $u=c_1e_1+c_2e_2+c_3e_3$.
Then $\mathrm{sec}(\s)=c_1^2\l_{23}+c_2^2\l_{13}+c_3^2\l_{12}.$

\end{lem}

\subsection{A Partition of $M$}

Fix $\e=-1, 0, 1$. Recall that $M$ is said to have higher rank if every geodesic $\g:\mathbb{R}\to M$ admits a parallel field $V(t)$ perpendicular to $\g'(t)$ such that
$\mathrm{sec}(V(t), \g'(t))\equiv \e$. 
The following pointwise version of the rank assumption was introduced in  \cite{MR3298721} to study the local structure of manifolds of higher rank. 

\begin{dfn}
A Riemannian manifold has constant vector curvature $\e$, denoted by $\mathrm{cvc}(\e)$,
if every tangent vector is contained in a curvature $\e$ plane.
\end{dfn}
Clearly, manifolds of higher rank have $\mathrm{cvc}(\e)$. For $\mathrm{cvc}(\e)$-manifolds, Lemma \ref{sectional} and (\ref{per}) imply that at least one of the $\l_{ij}$ equals $\e$.  Throughout the remainder of the paper, we assume that indices have been chosen such that
$\l_{12}=\e$. Letting $\l=\l_{13}$ and $\L=\l_{23}$, we then have
\begin{equation} \label{order}
\l_{13}=\l \leq \e \leq \L=\l_{23}.
\end{equation}

\begin{dfn}
A point $p\in M$ is said to be 
\begin{enumerate}
\item \textit{isotropic} if $\l=\L=\e$,
\item \textit{extremal} if precisely one of $\l$ or $\L$ equals $\e$,
\item \textit{generic} if $\l<\e<\L$.
\end{enumerate}
\end{dfn}
Throughout, let $\mathcal{I}, \mathcal{E}, \mathcal{O}$ denote the set of isotropic, extremal, and generic points respectively.
A three-manifold $M$ of higher rank is the disjoint union 
$$M=\mathcal{I}\cup \mathcal{E}\cup \mathcal{O}.$$ 
As the sectional curvature is constant on $\mathcal{I}$, rank rigidity occurs when $ \mathcal{E}=\mathcal{O}=\emptyset$.

Note that for each $p\in \mathcal{E}$, either $\l<\e=\L$ or $\l=\e<\L$ at $p$. This motivates the following partition of $\mathcal{E}$. 
\begin{dfn} \label{eplus_def}
Define subsets $\mathcal{E}_-$ and $\mathcal{E}_+$ of $\mathcal{E}$ by
$$\mathcal{E}_-=\{p\in \mathcal{E}\,|\, \l=\e\}$$ and
$$\mathcal{E}_+=\{p\in \mathcal{E}\,|\, \L=\e\}.$$
\end{dfn} 

 When $p$ is not an isotropic point, one can apply
Lemma \ref{sectional} to characterize the curvature $\e$ planes in $T_pM$ in terms of their normal vectors.
\begin {lem}  \label {lem1}

 Let $p\in\mathcal{E}\cup \mathcal{O}$ and let $\s$ be a 2-plane in $T_p M$ with unit normal vector $u=c_1e_1+c_2e_2+c_3e_3$. Then $\mathrm{sec}(\s)=\e$ if and only if 
 \begin{enumerate}
 \item $c_1=0$ when $p\in \mathcal{E}_-$,
 \item$c_2=0$ when $p\in \mathcal{E}_+$,
  \item $c_1=\pm mc_2$ when $p\in \mathcal{O}$, where $m=\sqrt \frac{\e-\l}{\L-\e}$. 
  \end{enumerate}
\end {lem}  

\begin{proof}
Immediate from Lemma \ref{sectional}.
\end{proof}

We define the following disjoint partition of $U_pM$.

\begin{dfn}
A vector $X\in U_pM$ is
\begin{enumerate}
\item \textit{isocurved} if every 2-plane in $T_pM$ containing $X$ has curvature $\e$,
\item \textit {unicurved} if $X$ is contained in only one curvature $\e$ 2-plane,
\item \textit{generic} if it is neither isocurved nor unicurved.
\end{enumerate}
\end{dfn}

\begin{dfn}
A geodesic $\g:(k, l)\to M$ is isocurved if $\g'(t)$ is isocurved for all $t\in(k,l)$.
\end{dfn}

Since the normal vectors of 2-planes containing $X$ form a two dimensional vector subspace, Lemma \ref {lem1} implies the 
following characterization of isocurved vectors.

\begin{cor} \label{isovector}
Let $p\in\mathcal{E}\cup \mathcal{O}$ and let $X=x_1e_1+x_2e_2+x_3e_3\in U_pM$. Then $X$ is isocurved if and only if
\begin{enumerate}
\item $x_1=\pm1$ if $p\in \mathcal{E}_-$,
\item $x_2=\pm1$ if $p\in \mathcal{E}_+$,
\item $x_3=0$ and $x_2 = \pm mx_1$ if $p\in \mathcal{O}$.
\end{enumerate}
\end{cor}

On the interior of generic set or on the interior of extremal set, it is possible to find a local orthonormal frame that satisfy (\ref{per}).
Let $\{e_1, e_2, e_3\}$ be such a local orthonormal frame. Throughout, we use the Christoffel symbol
$\G_{ij}^k$ to denote $<\n_{e_i} e_j, e_k>$.

\begin{lem}
For each $p\in \mathcal{O}$,
\begin{equation} \label{Bianchi_diag}
\G_{11}^3=m^2 \G_{22}^3
\end{equation}
\end{lem}
\begin{proof}
The lemma follows from (\ref{Ricci_diag}) and the differential Bianchi identity 
$$(\n_{e_3})R(e_1,e_2,e_2,e_1)+(\n_{e_1})R(e_2,e_3,e_2,e_1)+(\n_{e_2})R(e_3,e_1,e_2,e_1)=0.$$ 
\end{proof}

\begin{rem}\label{uniqueness}
As the eigenvalues of the Ricci tensor are distinct at points in $\mathcal{O}$, the eigenspces define three distinct \textit{global} smooth line fields on $\mathcal{O}$. Given a connected component $\mathcal{C}$ of $\mathcal{O}$, each of these line fields may or may not be orientable over $\mathcal{C}$. In the case where one of these line field is not orientable on $\mathcal{C}$, it resolves to an oriented line field on some double cover of $\mathcal{C}$.
As the rank assumption and finite volume assumption passes to finite covers, we may always assume that 
there exists a \textit{global} Ricci diagonalizing orthonormal frame $\{e_1, e_2, e_3\}$ on $\mathcal{C}$ (the finite volume
assumption is required for the proof of hyberbolic rank rigidity).
\end{rem}

\subsection{Manifolds of Extremal Curvature}

 
\begin{dfn}
A three-manifold $M$ of $\mathrm{cvc} (\e)$ is said to have extremal curvature if for every $p\in M$, $\mathrm{sec}_p\geq \e$ or $\mathrm{sec}_p\leq \e$. 
\end{dfn}

The definition is equivalent to $\mathcal{O}=\emptyset$, or equivalently, $M=\mathcal{E}\cup \mathcal{I}$. The following proposition, a structural result used in subsequent sections, is proved in \cite{MR0322740} or \cite[Corollary 2.10]{MR3298721}.

\begin{prop} \label{extreme_structure}
Let $M$ be a three-manifold of $\mathrm{cvc}(\e)$ and with extremal curvature. 
For each connected component of $\mathcal{C}$ of $\mathcal{E}$, there exists a complete geodesic field $E$ on $\mathcal{C}$ or a double cover of $\mathcal{C}$ whose integral curves are isocurved geodesics.
\end{prop}

The structural result leads to the following rigidity theorems proved in \cite[Theorem 1.2]{MR3298721} and \cite[Theorem 3.5]{MR3298721} respectively.
\begin{thm} \label{cvc}
A finite volume complete Riemannian three-manifold with extremal curvature -1 and higher hyperbolic rank is hyperbolic.
\end{thm}

\begin{thm} \label{cvc_hom}
Assume that $M$ is a connected, simply-connected, complete and homogeneous three-manifold with extremal curvature -1. If $M$ has higher hyperbolic rank, then 
$M$ is isometric to the three dimensional hyperbolic space. 
\end{thm}

\section{Local Structure in $\mathcal{O}$}  
This section studies the local structure of generic points in a three-manifold of higher rank. As the generic set $\mathcal{O}$ is open, we may consider a local orthonormal frame $\{e_1, e_2, e_3\}$ and satisfying (\ref{order}). With respect to this local orthonormal frame, 
\begin{equation}
\l<\e<\L. 
\end{equation}

Specifically, this section derives local equations satisfied by the Christoffel symbols with respect to this frame in a three-manifold of higher rank.
Roughly speaking, these equations will be derived as follows.

Lemma \ref{plane_o} below will show that each generic vector tangent to $\mathcal{O}$ belongs to exactly two curvature $\e$ planes. Generic vectors are characterized by an open condition on tangent vectors. 
Hence, given a geodesic with generic initial velocity vector, its velocity vectors remain generic for short time, and consequently, also belong to exactly two curvature $\e$ planes.  
As geodesics have higher rank, they are also contained in a \textit{parallel} family of curvature $\e$ planes. As a result, the parallel field equations apply to at least one of these 
two families of curvature $\e$ planes along generic geodesics. Essentially, these parallel field equations imply the desired Christoffel symbol relations.
 
\subsection{Pointwise Calculations}

Given a generic point $p\in \mathcal{O}$, each unit tangent vector $X\in U_pM$ may be expressed in the frame above as $X=x_1e_1+x_2e_2+x_3e_3$, where $x_1^2+x_2^2+x_3^2=1$. 
For simplicity, we will frequently write  $X=(x_1, x_2, x_3)$, omitting the frame vectors from our notation.


\begin{lem} \label{CS}
Let $p\in \mathcal{O}$. For $X=(x_1,x_2,x_3)\in U_pM$ a generic or unicurved unit vector, 
$$1+m^2-(m x_2-x_1)^2>0$$ and 
$$1+m^2-(-m x_2-x_1)^2>0.$$
\end{lem}
\begin{proof}
We prove the first inequality; the second inequality is obtained analogously after replacing $m$ by $-m$. 
Note that
\begin{equation} \label{cs}
(mx_2-x_1)^2=|(x_1, x_2)\cdot(-1,m)|^2 \leq (x_1^2+x_2^2)(1+m^2) \leq 1+m^2,
\end{equation}
where the first inequality is from Cauchy-Schwartz and the second inequality is from the fact that $X$ is a unit vector.
The statement of the lemma holds if one of these inequalities is strict.
The first inequality is strict unless $x_2=-mx_1$, and the second inequality is strict unless $x_3=0$. 
The lemma now follows from Corollary \ref{isovector}. 

\end{proof}

By Lemma \ref{CS}, we may associate to each generic or unicurved unit vector $X=(x_1,x_2,x_3)$ the numbers

\begin{equation} \label{K_def}
K=\sqrt{1+m^2-(m x_2-x_1)^2}
\end{equation}
and 
\begin{equation} \label{L_def}
L=\sqrt{1+m^2-(-m x_2-x_1)^2},
\end{equation}
and 
the following two unit vectors:
\begin {equation} \label {v_o}
V^+=\frac{1}{K}\big(1+mx_1x_2-x_1^2, mx_2^2-x_1x_2-m, (mx_2-x_1)x_3\big)
\end{equation}
and
\begin {equation}
\begin {split}
V^-=\frac{1}{L}\big(1-mx_1x_2-x_1^2, -mx_2^2-x_1x_2+m,(-mx_2-x_1)x_3\big).
\end{split}
\end{equation}
Direct computations show that $V^+, V^- \in U_pM$, and $V^+, V^- \perp X$. Note that the expression for $V^-$
can be obtained by changing every $m$ in the formula for $V^+$ to $-m$.


\begin {lem} \label {plane_o}
Let $p\in \mathcal{O}$ and let $X=(x_1, x_2,x_3)\in U_p M$ be a unicurved or generic unit vector and let $V=(v_1,v_2,v_3)$ be a unit vector perpendicular to $X$.
\begin {enumerate}

\item
If $x_3=0$, then $V^+= \pm V^-$. Moreover, $\mathrm{sec}(X,V)=\e$ if and only if $V=\pm V^+$.

\item 
If $x_3\neq0$, then $V^+$ and $V^-$ are linearly independent. Moreover, $\mathrm{sec}(X, V)=\e$ if and only if $V=\pm V^+$ or $V=\pm V^-$.
\end{enumerate}

\end{lem}

\begin {proof}
Let
$u=X \times V=(x_2v_3-x_3v_2, x_3v_1-x_1v_3, x_1v_2-x_2v_1)$. As $X$ and $V$ are unit vectors and $X\perp V$, $u$ is a unit 
normal vector to the plane spanned by $X$ and $V$.
By Lemma \ref{lem1}, $\mathrm{sec}(X, V)=\e$ if and only if
\begin {equation*} 
x_2v_3-x_3v_2=\pm m(x_1v_3-x_3v_1).
\end{equation*}
In other words, a unit vector $V=(v_1,v_2,v_3)$ satisfies $V\perp X$ and $\mathrm{sec}(X, V)=\e$ if and only if it satisfies
one of the following two linear systems:
\begin {equation} \label{plane_equation}
\begin{split}
&x_1v_1+x_2v_2+x_3v_3=0\\
& mx_3v_1-x_3v_2+(x_2-mx_1)v_3=0
\end{split}
\end{equation}
or
\begin {equation} \label{plane_equation_2}
\begin{split}
&x_1v_1+x_2v_2+x_3v_3=0\\
&- mx_3v_1-x_3v_2+(x_2+ mx_1)v_3=0.
\end{split}
\end{equation}

To prove the first assertion of (1), use the facts that $x_3=0$ and $x_1^2+x_2^2=1$ to show
$V^+=\pm (x_2, -x_1, 0)$
where $\pm$ is positive if and only if $x_2+mx_1>0$.
Similarly, $V^-=\pm  (x_2, -x_1, 0)$ where $\pm$ is positive if and only if $x_2-mx_1>0$.

To prove the second assertion of (1), note that when $x_3=0$ and $x_2\neq mx_1$, the system (\ref{plane_equation}) 
has a one dimensional solution set spanned by $V^+$. Similarly, when  $x_3=0$ and $x_2\neq -mx_1$ the system (\ref{plane_equation_2}) 
has a one dimensional solution set spanned by $V^-$.

To prove the second assertion of (2), note that system (\ref{plane_equation}) has a one dimensional solution since, for instance, 

\begin{equation} \label{matrix}
\begin{vmatrix} 
x_1 &x_2&x_3  \\
 mx_3 & -x_3 &x_2- mx_1\\
x_3& 0 &-x_1 
\end{vmatrix}=
x_3\neq0.
\end{equation}
Direct verification shows that the solution set is spanned by $V^+$. The same holds after replacing system (\ref{plane_equation}) with
(\ref{plane_equation_2}), $m$ with $-m$, and $V^+$ with $V^-$.

Finally, as
\begin{equation*} 
\begin{vmatrix} 
x_1 &x_2&x_3  \\
mx_3 & -x_3 &x_2-mx_1\\
-mx_3 & -x_3 &x_2+mx_1 
\end{vmatrix}=-2mx_3\neq0,
\end{equation*}
$V^+$ and $V^-$ are linearly independent, concluding the proof.
\end{proof}

\begin{rem} \label {genericity}
Lemma \ref{plane_o} implies that a vector $X=(x_1, x_2, x_3)\in U_pM$ is unicurved if and only if $x_3=0$ and $x_2 \neq \pm mx_1$, 
and is generic if and only if $x_3\neq0$. Furthermore, each generic vector is contained in exactly two curvature $\e$ planes.
\end{rem}

\subsection{Calculations Along Geodesics} 
Computations in this subsection are mostly done along geodesics. To be more precise,
let $\g:[0, \d]\to \mathcal{O}$ be a geodesic. We may write
$\g'(t)=(x_1(t), x_2(t), x_3(t))$. 
Quantities such as $V^+$ and $V^-$ defined on each tangent space now vary along the geodesic $\g(t)$. To simplify the notation,  we may not indicate 
that a certain quantity is a function of $t$. For example, we write $x_i$ instead of $x_i(t)$ and $m$ instead of $m(\g(t))$.

The reader may verify that 
\begin{equation} \label{K_prime}
K'=\frac{mm'-(mx_2-x_1)(mx_2-x_1)'}{K}
\end{equation} 
and
\begin{equation}
L'=\frac{mm'-(-mx_2-x_1)(-mx_2-x_1)'}{L}
\end{equation} 
where $K$ and $L$ are defined in (\ref{K_def}) and (\ref{L_def}).

\begin{lem} \label {parallel}
Let $\g:[0, \d]\rightarrow \mathcal{O}$ be a geodesic such that $\g'(t)$ is generic for all $t\in[0,\d]$. 
Then  either $V^+(t)$ or $V^-(t)$ is a parallel field along $\g$ for all 
$t\in [0, \d]$.
  
\end{lem} 

\begin{proof}
As $M$ has higher rank, there exists a unit normal parallel field $V(t)$ along $\g(t)$ such that 
$$sec(\g'(t), V(t))=\e \text{ } \forall t\in [0, \d].$$
As $\g'(t)$ is generic for each $t$,
part (2) of Lemma \ref{plane_o} implies that for each $\t\in[0, \d]$, $V(t)=\pm V^+(t)$ or $V(t)=\pm V^-(t)$.
By continuity, the lemma follows.

\end{proof}

\begin{lem} \label{computation}
Let $\g:[0, \d]\rightarrow \mathcal{O}$ be a geodesic whose velocity vectors have components $\g'(t)=(x_1(t), x_2(t), x_3(t))$.
Assume that $\g'(t)$ is generic for all $t \in [0, \d]$, then
$V^+(t)$ is parallel along $\g$ if and only if
\begin {equation} \label {computation_o}
\begin{split}
&\{m\G_{11}^3-m^2\G_{12}^3\}x_1^2\\
+&\{\G_{21}^3-m\G_{22}^3\}x_2^2\\
+&\{e_3(m)-\G_{31}^2\}x_3^2\\
+&\{\G_{11}^3-m\G_{12}^3+m\G_{21}^3-m^2\G_{22}^3\}x_1x_2\\
+&\{e_1(m)-(1+m^2)\G_{11}^2-m\G_{33}^1+m^2\G_{33}^2\}x_1x_3\\
+&\{e_2(m)+(1+m^2)\G_{22}^1-\G_{33}^1+m\G_{33}^2\}x_2x_3=0
\end{split}
\end{equation}
along $\g(t)$. 
\end{lem}
\begin{proof}

Writing $V^+(t)=(v_1(t),v_2(t),v_3(t))$ along $\g(t)$, we have that 
\begin{equation}
\begin{split} \label{v_def}
v_1&=\frac{1+mx_1x_2-x_1^2}{K}\\
v_2&=\frac{ mx_2^2-x_1x_2-m}{K}\\
v_3&=\frac{(mx_2-x_1)x_3}{K}
\end{split}
\end{equation}
by (\ref{v_o}). 

Direct computation using (\ref{K_def}) and (\ref{v_def}) shows that 
\begin{equation} \label {K_v}
v_1-mv_2=K.
\end{equation}
The vector field $V^+(t)$ is parallel if and only if it satisfies the parallel field equations:
\begin{equation} \label{parallel_eq}
0=\dfrac{DV^+}{dt}=\sum \limits_{k=1}^3(\dfrac{dv_k}{dt}+\sum \limits_{i,j} \G_{ij}^k x_iv_j)e_k
\end{equation}
where $i,j\in \{1,2,3\}$. In what follows, we prove the lemma by showing that equations (\ref{parallel_eq}) hold if and only if 
(\ref{computation_o}) holds.

As a first step, substitute (\ref{v_def}) into the undifferentiated terms in (\ref{parallel_eq}) to obtain


\begin{equation} \label{parallel_2}
\begin{split}
v_1'&+\dfrac{1}{K}\Big\{(mx_2-x_1)(\sum_{i,j}\G_{ij}^1x_ix_j)-m\sum_{i=1}^3 x_i\G_{i2}^1\Big\}=0,\\
v_2'&+\dfrac{1}{K}\Big\{(mx_2-x_1)(\sum_{i,j}\G_{ij}^2x_ix_j)+\sum_{i=1}^3 x_i\G_{i1}^2\Big\}=0,\\
v_3'&+\dfrac{1}{K}\Big\{(mx_2-x_1)(\sum_{i,j}\G_{ij}^3x_ix_j)+\sum_{i=1}^3 x_i(\G_{i1}^3-m\G_{i2}^3)\Big\}=0.
\end{split}
\end{equation}

As $\g$ is a geodesic, it satisfies the geodesic equations
\begin {equation} \label {geod}
\sum \limits_{k=1}^3(\dfrac{dx_k}{dt}+\sum \limits_{i,j} \G_{ij}^k x_ix_j)e_k=0.
\end{equation}

Next, use (\ref{geod}) in (\ref{parallel_2}) to obtain  

\begin{equation} \label{computation_1}
v_1 '-\dfrac{m(\sum_{i=1}^3x_i \G_{i2}^1)}{K}-\frac{mx_2-x_1}{K}x_1'=0,
\end{equation}

\begin{equation} \label{computation_2}
v_2 ' +\dfrac{(\sum_{i=1}^3x_i \G_{i1}^2)}{K}-\frac{mx_2-x_1}{K}x_2'=0,
\end{equation}
and
\begin{equation} \label{computation_3}
v_3' +\frac{\sum_{i=1}^3x_i \G_{i1}^3-m\sum_{i=1}^3x_i \G_{i2}^3}{K}
-\frac{mx_2-x_1}{K}x_3'=0.
\end{equation}

Now we claim that equations (\ref{computation_1}) and (\ref{computation_2}) are equivalent. Indeed, subtract
$m$ times the left hand side of (\ref{computation_2}) from the left hand side of (\ref{computation_1}), and use 
equality (\ref{K_v}) to obtain zero as follows:

\begin{equation*}
\begin{split}
&(v_1'-mv_2')-\dfrac{mx_2-x_1}{K}(x_1'-mx_2')\\
&=(v_1-mv_2)'+m'v_2-\dfrac{mx_2-x_1}{K}(x_1'-mx_2')\\
&=K'+m'v_2-\dfrac{mx_2-x_1}{K}(x_1'-mx_2')\\
&=\dfrac{1}{K}\{mm'-(mx_2-x_1)(mx_2-x_1)'\}+m'v_2+\dfrac{mx_2-x_1}{K}(mx_2'-x_1')\\
&=0.
\end{split}
\end{equation*}
Since $m\neq 0$, this claim follows. 

Next, we will show that (\ref{computation_2}) and (\ref{computation_3}) are equivalent. 
To prove this claim, first differentiate the $v_i$ using (\ref{K_prime}) and then substitute into (\ref{computation_2}) and (\ref{computation_3}) to obtain

\begin {equation} \label{comp_1}
\begin{split}
&\dfrac{1}{K^3}\big\{(x_1^2+x_2^2-1)m'+(mx_2'-x_1')(x_2+mx_1)\big\}\\
&+\dfrac{1}{K}(\sum_{i=1}^3x_i \G_{i1}^2)=0,
\end{split}
\end{equation}
and
\begin {equation} \label{comp_2}
\begin{split}
&\dfrac{1}{K^3}\big\{(x_2+mx_1)x_3m'+(1+m^2)(mx_2'-x_1')x_3\big\}\\
&+\dfrac{1}{K}(\sum_{i=1}^3x_i \G_{i1}^3-m\sum_{i=1}^3x_i \G_{i2}^3)=0.
\end{split}
\end{equation}

Use (\ref{geod}) to show that 
\begin{equation}\label{geod_2}
mx_2'-x_1'=-(x_2+mx_1)(\sum_{i=1}^3x_i \G_{i1}^2)-x_3(\sum_{i=1}^3x_i \G_{i1}^3-m\sum_{i=1}^3x_i \G_{i2}^3).
\end{equation}
Substitute (\ref{geod_2}) into (\ref{comp_1}) and (\ref{comp_2}) to obtain
\begin {equation} \label{computation_4}
\begin{split}
\dfrac{1}{K^3}&\{x_3^2m'-x_3^2(1+m^2)(\sum_{i=1}^3x_i \G_{i1}^2)\\
+&x_3(x_2+mx_1)(\sum_{i=1}^3x_i \G_{i1}^3-m\sum_{i=1}^3x_i \G_{i2}^3)\}=0,
\end{split}                 
\end{equation}
and
\begin {equation} \label{computation_5}
\begin{split}
\dfrac{1}{K^3}&\{x_3(x_2+mx_1)m'-x_3(x_2+mx_1)(1+m^2)(\sum_{i=1}^3x_i \G_{i1}^2)\\
+&(x_2+mx_1)^2(\sum_{i=1}^3x_i \G_{i1}^3-m\sum_{i=1}^3x_i \G_{i2}^3)\}=0.
\end{split}                 
\end{equation}
Observe that equation (\ref{computation_4}) and (\ref{computation_5}) differ by the nonzero factor $\frac{x_3}{x_2+mx_1}$,
concluding the proof of the claim.

Up to now, we have shown that the parallel field equations are equivalent to (\ref{computation_5}). 
Since $K\neq0$ and 
$\g'(t)$ is generic, (\ref{computation_5}) is equivalent to
\begin {equation} \label{computation_6}
\begin{split}
&x_3\{(m'-(1+m^2)(x_1\G_{11}^2+x_2\G_{21}^2+x_3\G_{31}^2)\\
&+(x_2+mx_1)\G_{31}^3-m(x_2+mx_1)\G_{32}^3)\}\\
&+(x_2+mx_1)(x_1\G_{11}^3+x_2\G_{21}^3-mx_1\G_{12}^3-mx_2\G_{22}^3)=0.
\end{split}
\end{equation}

To conclude the proof, note that (\ref{computation_o}) follows from (\ref{computation_6}), since $m'=x_1e_1(m)+x_2e_2(m)+x_3e_3(m)$.

\end{proof}

\begin{rem} \label{minus_equivalence}
With the same assumption as in Lemma \ref{computation}, an analogous proof replacing $m$ with $-m$ and $K$ with $L$ shows that $V^-(t)$ is parallel if and only if 
\begin {equation} \label {computation_minus}
\begin{split}
&\{-m\G_{11}^3-m^2\G_{12}^3\}x_1^2\\
+&\{\G_{21}^3+m\G_{22}^3\}x_2^2\\
+&\{-e_3(m)-\G_{31}^2\}x_3^2\\
+&\{\G_{11}^3+m\G_{12}^3-m\G_{21}^3-m^2\G_{22}^3\}x_1x_2\\
+&\{-e_1(m)-(1+m^2)\G_{11}^2+m\G_{33}^1+m^2\G_{33}^2\}x_1x_3\\
+&\{-e_2(m)+(1+m^2)\G_{22}^1-\G_{33}^1-m\G_{33}^2\}x_2x_3=0
\end{split}
\end{equation}
along $\g(t)$.
\end{rem}

By Remark \ref{genericity}, generic tangent vectors form an open subset in the generic set $\mathcal{O}$. Hence, for $p\in \mathcal{O}$ and $X=(x_1, x_2, x_3)\in U_pM$ a generic vector, there exists $\d>0$ such that the velocity vectors of the geodesic $\g(t)=\exp_p tX$ are generic for all $t\in[0, \d]$. By Lemmas \ref{parallel} and \ref{computation}, it follows that either (\ref{computation_o}) or (\ref{computation_minus}) hold along $\g(t)$ on $[0, \d]$. In particular, for every $p\in \mathcal{O}$
and generic vector $X\in U_pM$, either (\ref{computation_o})
or (\ref{computation_minus}) holds for $X$ at $p$. This motivates the following definition.

\begin{dfn} \label{bplus}
Define subsets $\mathit{B}_p^+$ and $\mathit{B}_p^-$ of $U_p M$ as follows.
\begin{enumerate}
\item $X\in \mathit{B}_p^+$ if and only if (\ref{computation_o}) holds for $X$ at $p$.
\item $X\in \mathit{B}_p^-$ if and only if (\ref{computation_minus}) holds for $X$ at $p$.
\end{enumerate}
\end{dfn}

\begin{lem} \label {limit}
For each $p\in \mathcal{O}$, $U_pM \subset \mathit{B}_p^+\cup\mathit{B}_p^-$.
\end{lem}

\begin{proof}
Note that the set $\mathit{B}_p^+$ and $\mathit{B}_p^-$ are closed subsets of $U_pM$. By Remark \ref{genericity}, the generic vectors 
are dense in $U_pM$. Hence the remarks preceding Definition \ref{bplus} and the pigeon-hole principle imply the lemma.   

\end{proof}


\begin {prop}\label{point}
For each $p\in \mathcal{O}$, at least one of the following holds:
\begin{equation} \label{point_p}
\begin{split}
\G_{11}^3&=m\G_{12}^3=m\G_{21}^3=m^2\G_{22}^3\\
e_1(m)&=(1+m^2)\G_{11}^2+m\G_{33}^1-m^2\G_{33}^2\\
e_2(m)&=-(1+m^2)\G_{22}^1+\G_{33}^1-m\G_{33}^2\\
e_3(m)&=(1+m^2)\G_{31}^2,
\end{split}
\end{equation}
or
\begin{equation} \label{point_m}
\begin{split}
\G_{11}^3&=-m\G_{12}^3=-m\G_{21}^3=m^2\G_{22}^3\\
e_1(m)&=-(1+m^2)\G_{11}^2+m\G_{33}^1+m^2\G_{33}^2\\
e_2(m)&=(1+m^2)\G_{22}^1-\G_{33}^1-m\G_{33}^2\\
e_3(m)&=-(1+m^2)\G_{31}^2.
\end{split}
\end{equation}

\end{prop}

As $\mathit{B}_p^+$ and $\mathit{B}_p^-$ are closed subsets of $U_pM$, Lemma \ref{limit} implies that 
at least one of them contains an interior. Hence to prove the Proposition \ref{point}, it suffices to prove the following lemma.

\begin{lem} \label{argument_three}
Let $p\in \mathcal{O}$. The equation (\ref{point_p}) (respectively (\ref{point_m})) holds  at $p$  if $\mathit{B}_p^+$  (respectively  $\mathit{B}_p^-$) contains an interior. 
\end{lem}

\begin{proof}
We prove the case when $\mathit{B}_p^+$ contains an interior. The proof for the case when $\mathit{B}_p^-$ contains an interior is analogous. 
If $\mathit{B}_p^+$ contains an open set in $U_pM$, then the degree two homogeneous polynomial  in variables $x_1, x_2, x_3$ defined by the left hand side of (\ref{computation_o}) vanishes identically on an open subset of $T_pM$. Hence its coefficients all vanish.
Consider the coefficients for $x_1x_3$, $x_2x_3$ and $x_3^2$, we obtain the derivative equations in (\ref{point_p}).
As $m\neq0$, by considering the coefficients for $x_1^2$ and $x_2^2$, we have that 
\begin{equation}\label{Christoffel_1}
\G_{11}^3=m\G_{12}^3
\end{equation}
and
\begin{equation}\label{Christoffel_2}
\G_{21}^3=m\G_{22}^3.
\end{equation}
Now, (\ref{Bianchi_diag}), (\ref{Christoffel_1}), and (\ref{Christoffel_2}) imply that  
\begin{equation}\label{Christoffel_relations}
\G_{11}^3=m\G_{12}^3=m\G_{21}^3=m^2\G_{22}^3,
\end{equation}
as desired.

\end{proof}

\section {Local Structure on the interior of $\mathcal{E}$} \label {lse}

This section derives relations between Christoffel symbols at interior points to the extremal set $\mathcal{E}$. 
Let $p\in \mathcal{E}$ and let $\{e_1, e_2, e_3\}$ be a Ricci diagonalizing orthonormal frame at $p$ that satisfies (\ref{order}).

Recall that from Definition \ref{eplus_def}, we had a disjoint partition $\mathcal{E}=\mathcal{E}_+\cup \mathcal{E}_-$. We first study the case when $p\in \mathcal{E}_-$. The following lemma will show that a vector $X=(x_1, x_2, x_3)$ is unicurved if $x_1\neq \pm 1$.

\begin {lem} \label{pointwise_e}
Let $p\in \mathcal{E}_-$ and let $X=(x_1, x_2,x_3)\in U_pM$ with $x_1\neq \pm1$.
If $V\in U_pM$ is perpendicular to $X$, then $\mathrm{sec}(X,V)=\e$ if and only if 
\begin {equation} \label {v_e}
V=\pm \Big( \frac{1-x_1^2}{\sqrt{1-x_1^2}},
       \frac{-x_1x_2}{\sqrt{1-x_1^2}},
       \frac{-x_1x_3}{\sqrt{1-x_1^2}}\Big).
\end{equation} 
\end{lem}

\begin {proof}
Write $V=(v_1,v_2,v_3)$. We argue like Lemma \ref{plane_o}.
The unit vector
$$u=X \times V=(x_2v_3-x_3v_2, x_3v_1-x_1v_3, x_1v_2-x_2v_1)$$ 
is normal to the plane spanned by $X$ and $V$.

By Lemma \ref{lem1}, $\mathrm{sec}(X, V)=\e$ if and only if
$$x_2v_3-x_3v_2=0.$$
Since $V\perp X$, $\mathrm{sec}(X, V)=\e$ if and only if the following system of equations hold:
\begin{equation*} 
(v_1, v_2, v_3)\cdot (0,-x_3,x_2)=0
\end{equation*}
and
\begin{equation*} 
(v_1, v_2, v_3)\cdot (x_1, x_2, x_3)=0,
\end{equation*}
or equivalently, if and only if
\begin {equation*}
\begin {split}
V&=\pm \dfrac{(x_1, x_2, x_3)\times(0, -x_3, x_2)}{|(x_1, x_2, x_3)\times(0, -x_3, x_2)|}\\ 
  &=\pm (\frac{1-x_1^2}{\sqrt{1-x_1^2}}, \frac{-x_1x_2}{\sqrt{1-x_1^2}}, \frac{-x_1x_3}{\sqrt{1-x_1^2}}).
\end{split}
\end{equation*}

\end{proof}

\begin{rem} \label{vector_character}
Lemma \ref{pointwise_e} and Corollary \ref{isovector} together show that each vector tangent to $\mathcal{E}_-$ is either isocurved or univurved. 
A vector $X=(x_1, x_2, x_3)$ is unicurved if and only if $x_1\neq \pm 1$, and is isocurved if and only if $x_1=\pm1$. 
\end{rem}

Notice that the $V$ in (\ref{v_e}) can be obtained by taking $m=0$ in (\ref{v_o}). Likewise, setting $m=0$ in the statement and proof of Lemma \ref{computation}
yields the following:

\begin{lem} \label{parallel_e}
Let $\g:[0, \d]\rightarrow \mathcal{E}_-$ be a geodesic whose velocity vectors have components $\g'(t)=(x_1(t), x_2(t), x_3(t))$.
If $\g'(t)$ is unicurved for all $t \in [0, \d]$, then
\begin {equation*}
V(t)=(\frac{1-x_1^2}{\sqrt{1-x_1^2}},\frac{-x_1x_2}{\sqrt{1-x_1^2}},\frac{-x_1x_3}{\sqrt{1-x_1^2}})
\end{equation*}
is parallel along $\g$ if and only if 
\begin{equation}\label{parallel_e_chris}
x_3\{-x_1\G_{11}^2-x_2\G_{21}^2-x_3\G_{31}^2+x_2\G_{31}^3\}+x_2\{x_1\G_{11}^3+x_2\G_{21}^3\}=0.
\end{equation}
\end{lem}

\begin {prop} \label {point_e}
The following is true for each interior point of $\mathcal{E}_-$:
$$\G_{11}^2=\G_{11}^3=\G_{21}^3=\G_{31}^2=0$$
and
$$\G_{33}^1=\G_{22}^1.$$
\end{prop}

\begin {proof}
Let $p$ be an interior point of $\mathcal{E}_-$.
Then each geodesic starting from $p$ remains in $\mathcal{E}_-$ for a short time. Moreover, by Remark \ref{vector_character}, the unicurved vectors form an 
open subset of unit vectors tangent to $M$ near $p$. Hence, a geodesic starting from $p$ whose initial velocity is unicurved,
has velocity vectors that remain unicurved for a short time.

Let $X=(x_1, x_2, x_3)$ be a unicurved vector. We claim that 
\begin{equation} \label{ex}
x_3\{-x_1\G_{11}^2+x_2(\G_{22}^1-\G_{33}^1)-x_3\G_{31}^2\}
+x_2\{x_1\G_{11}^3+x_2\G_{21}^3\}=0
\end{equation}
at the point $p$.

To prove the claim, let $\g(t)=\exp_p(tX)$.
As $M$ has higher rank, there exists a unit normal parallel field $W(t)$ along $\g(t)$ with $\mathrm{sec}(\g'(t), W(t))\equiv \e$.
By Lemmas  \ref{pointwise_e} and  \ref{parallel_e}, $W(t)=\pm V(t)$ for a short time and the claim follows.

As the unicurved vectors form an open subset of $U_pM$, the degree 2 homogeneous polynomial in variables $x_1, x_2, x_3$ defined by the left hand side of (\ref{ex}) vanishes identically on an open subset of $T_pM$. Hence its coefficients all equal zero, concluding the proof.

\end{proof}

Note that only the assumption that $\L\neq\e$ was used to prove Proposition \ref{point_e}. Hence, for $p\in \mathcal{E}_+$, the same proposition holds
after permuting the indices 1 and 2. We record this in the following proposition.

\begin{prop}  \label {point_e_2}
The following is true for each interior point of $\mathcal{E}_+$:
$$\G_{22}^1=\G_{22}^3=\G_{12}^3=\G_{32}^1=0$$
and
$$\G_{33}^2=\G_{11}^2.$$

\end{prop}

\section {Evolution Equations Along Isocurved Geodesics in $\mathcal{O}$ } 


Assume that the open set $\mathcal{O}$ is nonempty.  The goal of this section is to study the global structure of $\mathcal{O}$ by the Christoffel symbol relations given in Proposition \ref{point}.

 The global structure of $\mathcal{O}$ is described in terms of evolution equations of geometric quantities along a specific family of
isocurved geodesics. We begin this section with a construction of such a family of isocurved geodesics. 
\subsection {Existence of Isocurved Geodesics}
Let $\mathcal{C}$ be a connected component of $\mathcal{O}$. Without loss of generality,
we may assume that there exists a global framing $\{e_1, e_2, e_3\}$ over $\mathcal{C}$ satisfying  (\ref{order}) as in Remark \ref{uniqueness}. With respect to this framing, at each point $p\in \mathcal{C}$ either (\ref{point_p}) or (\ref{point_m}) holds by Proposition \ref{point}.

\begin{dfn} \label{F_def}
Define closed subsets $\mathcal{F}^+$ and $\mathcal{F}^-$ of
$\mathcal{C}$ by
$$\mathcal{F}^+=\big\{p\in \mathcal{C}\,|\, \text{ (\ref{point_p}) holds} \big\}$$ 
and
$$\mathcal{F}^-=\big\{p\in \mathcal{C}\,|\, \text{ (\ref{point_m}) holds} \big\}.$$
\end{dfn}
Note that $\mathcal{F}^+\cup \mathcal{F}^-=\mathcal{C}$.

\begin{rem} \label{frame_switching}
The definition of $\mathcal{F}^+$ and $\mathcal{F}^-$ depends on the Ricci diagonalizing orthonormal frame. 
The reader may verify that points in $\mathcal{F}^+$ and points in $\mathcal{F}^-$ are switched if the orthonormal frame is switched from
$\{e_1, e_2, e_3\}$ to $\{e_1, -e_2, e_3\}$. 
\end{rem}

\begin{lem} \label{alternative}
After possibly switching the frame $\{e_1, e_2, e_3\}$ to $\{e_1, -e_2, e_3\}$, we may assume that one of the following holds:
\begin{enumerate}
\item $\mathcal{F^+}=\mathcal{F}^-=\mathcal{C}$,
\\or
\item The set  $\mathcal{C} \setminus \mathcal{F}^-$ contains an open set $U$ in $\mathcal{C}$.
\end{enumerate}
\end{lem}
\begin{proof}
Note that the sets $\mathcal{C}\setminus \mathcal{F}^-$ and $\mathcal{C} \setminus \mathcal{F}^+$ are open in $\mathcal{C}$. If both of these sets are empty, then (1) holds. Otherwise,
by Remark \ref{frame_switching}, after possibly switching the frame $\{e_1, e_2, e_3\}$ to $\{e_1, -e_2, e_3\}$, (2) holds.

\end{proof}

\begin{conv}
From now on, we always assume that the Ricci Diagonalizing frame is selected such that either (1) or (2) in Lemma \ref{alternative} holds.
\end{conv}

\begin{dfn}
Define two vector fields on $\mathcal{C}$ by
$$E^+=\frac {1}{\sqrt{1+m^2}}e_1-\frac{m}{\sqrt{1+m^2}}e_2,$$
$$E^-=\frac {1}{\sqrt{1+m^2}}e_1+\frac{m}{\sqrt{1+m^2}}e_2.$$
\end{dfn}

\begin{rem} \label{isocurved_property}
Note that these vector fields are isocurved by Corollary \ref{isovector}. They are characterized projectively by
having the property that every plane containing one of these vectors has curvature $\e$.
\end{rem}

\begin {lem} \label {geod_o}
On $\mathcal{F}^+$, $\n_{E^+} E^+=0$.
\end {lem}

\begin {proof}
On $\mathcal{C}$, direct calculations shows 
\begin {equation*}
\begin{split}
<\n_{E^+} E^+, e_1 >
  =&\frac{1}{\sqrt{1+m^2}}\{e_1(\frac{1}{\sqrt{1+m^2}})-me_2(\frac{1}{\sqrt{1+m^2}})\\
  -&\frac{m}{\sqrt{1+m^2}} \G_{12}^1+\frac{m^2}{\sqrt{1+m^2}}\G_{22}^1\},
\end{split}
\end{equation*}

\begin {equation*}
\begin{split}
<\n_{E^+} E^+, e_2 >
  =&\frac{1}{\sqrt{1+m^2}}\{e_1(\frac{-m}{\sqrt{1+m^2}})+me_2(\frac{m}{\sqrt{1+m^2}})\\
  +&\frac{1}{\sqrt{1+m^2}} \G_{11}^2-\frac{m}{\sqrt{1+m^2}}\G_{21}^2\},
\end{split}
\end{equation*}
and
\begin {equation*}
\begin{split}
<\n_{E^+} E^+, e_3 >
  =&\dfrac{1}{1+m^2} \{\G_{11}^3-m\G_{21}^3-m\G_{12}^3+m^2 \G_{22}^3\}.
\end{split}
\end{equation*}
Direct verification using  (\ref{point_p}) shows that each of these components is zero, concluding the proof.

\end{proof}


When $\mathcal{F}^+=\mathcal{F}^-=\mathcal{C}$, Lemma \ref{geod_o} implies that the integral curves of $E^+$ are geodesics.
Moreover, by part (3) of Corollary \ref{isovector}, these geodesics are isocurved.

In what follows, we construct a family of isocurved geodesics when 
there exists an open set  $U\subset \mathcal{C}\setminus \mathcal{F}^-$. To prove the desired result, we first prove the lemma below.




\begin{lem} \label{rank2_assist}
Let $p \in  \mathcal{C}\setminus \mathcal{F}^-$ and let $X=(x_1, x_2, x_3)\in U_pM$ be a generic vector.
Let $V^+$ be the vector perpendicular to $X$ defined by (\ref{v_o}). Then along $\g_X(t)=\exp_p (tX)$, 
the parallel vector field $V(t)$ determined by $V(0)=V^+$ satisfies 
$\mathrm{sec}(V(t), \g'_X(t))\equiv \e$.
\end{lem}

\begin{proof}
As $p\in \mathcal{C}\setminus \mathcal{F}^-$, Lemma \ref{argument_three} implies that $\mathit{B}_p^-$ has no interior in $U_pM$.
By continuity, it suffices to prove the lemma for generic vectors $X \in U_pM\setminus \mathit{B}_p^-$.

Let $X=(x_1, x_2, x_3)\in U_pM\setminus \mathit{B}_p^-$ be a generic vector.
As $M$ has higher rank, there exists a normal parallel field $V(t)$ along $\g_X(t)$ such that $\mathit{sec}(V(t), \g'_X(t))\equiv \e$.
Since $X$ is generic, Lemma \ref{parallel} implies that $V(t)=V^+(t)$ or $V(t)=V^-(t)$ for a short time. If $V(t)=V^-(t)$ for a short time, then by Remark \ref {minus_equivalence}, (\ref{computation_minus}) holds for $X$ at $p$, contrary to the assumption that $X\notin\mathit{B}_p^-$. 

\end{proof}
 
\begin{prop} \label{rank2_prop}
Let $p\in \mathcal{C}\setminus \mathcal{F}^-$. Then the geodesic $\a (t)=\exp_p (tE^+)$ is isocurved.
\end{prop}

\begin{proof}
As generic vectors are dense amongst the unit vectors perpendicular to $E^+$, it suffices to prove that each 
generic unit vector $W\perp E^+$ generates a parallel vector field $W(t)$ along $\a(t)$ satisfying $\mathrm{sec}(\a'(t), W(t))\equiv \e$.

Given $W\in U_pM$ as above, consider the great circle 
\begin{equation} \label{rank2_1}
c(s)=\cos{(s)} E^+ +\sin {(s)} W.
\end{equation}
For each small $s\geq0$, let $\g_s(t)=\exp_p (t c(s))$, and note that $\a(t)=\g_0(t)$ and that $W=c'(0)$.
Hence, by continuity, it suffices to prove that $c'(s)$ generates a parallel vector field $W_s(t)$ along $\g_s(t)$
satisfying $\mathrm{sec}(\g_s'(t), W_s(t))\equiv \e$.

Note that direct computation after substituting (\ref{rank2_1}) into (\ref{v_o}) shows that $c'(s)$ is $-V^+$ associated to $c(s)$. 
The proposition now follows from Lemma \ref{rank2_assist}.


\end{proof}

\begin{cor} \label{isocurved_geo_velocity}
Let $p\in \mathcal{C}\setminus \mathcal{F}^-$ and let $\a(t)=\exp_p (tE^+)$. Let $I\subset \mathbb{R}$ be the largest connected interval containing zero such that $\a(I)\subset \mathcal{C}$. Then on $I$, $\a'=E^+$. 
\end{cor}

\begin{proof}
Consider the subset $T\subset I$ defined by
$$T=\{t\in I\,|\, \a'(t)=E^+(\a(t))\}.$$
By construction, $0\in T$, and $T$ is closed. As $I$ is connected, it remains to prove that $T$ is open.

For $s\in T$, $\a'(t)$ is isocurved for each $t$ in a neighborhood of $s$ by Proposition \ref{rank2_prop}. The lemma now follows by the
continuity of the velocity vectors $\a'(t)$ and Remark \ref{isocurved_property}.

\end{proof}
 
We summarize the results obtained in this subsection in the following proposition.
\begin{prop} \label{tube}
There exists an open subset $U\subset \mathcal{C}$ with the property that each integral curve of 
$E^+$ with initial point in $U$ is an isocurved geodesic.

\end{prop}
\begin{proof}
As $\n_{E^+} E^+=0$ on $\mathcal{F}^+$, the proposition is obvious when (1) holds in Lemma \ref{alternative}.
When (2) in Lemma \ref{alternative} occurs, take $U$ to be an open subset of $\mathcal{C}\setminus \mathcal{F}^-$, and the proposition follows from Proposition \ref{rank2_prop} and Corollary \ref{isocurved_geo_velocity}.

\end{proof}

\subsection{Evolution Equations}

By Proposition \ref{tube}, there exists a family of isocurved geodesics in $\mathcal{C}$ that are tangent to $E^+$. In this subsection,
let $\a$ be one of these geodesics constructed in Proposition \ref {tube}. 
The goal of this subsection is to 
derive first order differential equations satisfied by geometric quantities along $\a$ in Proposition \ref{evo_tr} below.
The geometric quantities are described in terms of a new framing over $\mathcal{C}$.

\begin{dfn}
Define another orthonormal frame $\{E_1, E_2, E_3\}$ on $\mathcal{C}$ by
$E_1=E^+$, $E_3=e_3$, and
$$E_2=\frac {m}{\sqrt{1+m^2}}e_1+ \frac{1}{\sqrt{1+m^2}}e_2.$$

\end{dfn}
By Remark \ref{uniqueness}, the orthonormal frame $\{E_1, E_2, E_3\}$ is not Ricci diagonalizing. In what follows below, we reserve the 
Christoffel symbol notation $\G_{ij}^k$ for $<\n_{e_i} e_j, e_k>$.

\begin{lem} \label{x_lemma}
With respect to $\{E_1, E_2, E_3\}$, at least one of $R(E_1, E_2, E_2, E_3)$ or $R(E_1, E_3, E_3,E_2)$ is nonzero, and
\begin{equation} \label{x}
 R(E_1, E_2, E_3, E_1)=0.
\end{equation}

\end{lem}

\begin{proof}
As $E_1$ is isocurved, for every $a, b$ such that $a^2+b^2=1$,
\begin{equation*}
\begin{split}
\e&=R(E_1, a E_2+b E_3, a E_2+b E_3, E_1)\\
&= a^2R(E_1, E_2, E_2, E_1)+b^2R(E_1, E_3, E_3, E_1)+2ab R(E_1, E_2, E_3, E_1)\\
&= \e+2ab R(E_1, E_2, E_3, E_1).
\end{split}
\end{equation*}
The lemma now follows from (\ref{Ricci_diag}), since $\{E_1, E_2, E_3\}$ does not diagonalize Ricci.

\end{proof}

Note that by Proposition \ref{tube},
\begin{equation} \label{X_is_geodesic_field}
\n_{E_1} E_1=0
\end{equation}
along $\a$.
By using (\ref{X_is_geodesic_field}), the following Christoffel symbols vanish when the geodesic $\a$ enters $\mathcal{F}^-$. 

\begin{lem} \label{integrability_lem}
If $\a(t) \in \mathcal{F}^-$, then
\begin{equation} \label{Integrability}
\G_{22}^3=\G_{12}^3=\G_{21}^3=\G_{11}^3=0.
\end{equation} 

\end{lem}

\begin{proof}
First calculate 
\begin{equation} \label{parallel_fs_2}
<\n_{E_1} E_1, E_3 >=\dfrac{1}{1+m^2} \{\G_{11}^3-m\G_{21}^3-m\G_{12}^3+m^2 \G_{22}^3\}.
\end{equation}
Next, use (\ref{point_m}) to reduce (\ref{parallel_fs_2}) to 
\begin{equation}
<\n_{E_1} E_1, E_3 >=\dfrac{1}{1+m^2}\{4m^2\G_{22}^3\}.
\end{equation}

By (\ref{X_is_geodesic_field}) and the fact that $m$ is nonzero, it follows that $\G_{22}^3=0$.
This, upon substitution into (\ref{point_m}), shows that
\begin{equation*} 
\G_{22}^3=0=\G_{12}^3=\G_{21}^3=\G_{11}^3.
\end{equation*} 

\end{proof}

\begin {lem} \label {parallel_lem}
Along the geodesic $\a(t)$, we have that
\begin{equation} \label {parallel_fs}
\n_{E_1} E_2=\n_{E_1} E_3=0
\end{equation}
and
\begin{equation} \label{vanish_diagonal}
<\n_{E_2} E_1, E_3>=0.
\end{equation}

\end{lem}

\begin{proof}
We first prove (\ref{parallel_fs}).
By  (\ref{X_is_geodesic_field}) and the fact that
$\{E_1, E_2, E_3\}$ is an orthonormal basis, it suffices to prove that $\n_{E_1} E_3=0$.

Straightforward computation shows that 
\begin {equation} \label{parallel_fs_1}
\n_{E_1} E_3=\dfrac{1}{1+m^2}\{(\G_{13}^1-m\G_{23}^1)e_1+(\G_{13}^2-m\G_{23}^2)e_2\}.
\end{equation}
If $\a(t)\in \mathcal{F}^+$, substitute (\ref{point_p}) into (\ref{parallel_fs_1}) to obtain
$\n_{E_1} E_3=0.$
If $\a(t) \in \mathcal{F}^-$, substitute (\ref{Integrability}) into (\ref{parallel_fs_1}) to obtain
$\n_{E_1} E_3=0.$

To prove (\ref{vanish_diagonal}), first calculate
\begin{equation}
<\n_{E_2} E_1, E_3>=\frac{1}{1+m^2}\{m\G_{11}^3-m^2\G_{12}^3+\G_{21}^3-m\G_{22}^3\}.
\end{equation}

The equation (\ref{vanish_diagonal}) now follows from (\ref{point_p}) when in $\a(t) \in \mathcal{F}^+$, and from (\ref{Integrability}) when $\a(t) \in \mathcal{F}^-$.

\end{proof}

\begin{lem}\label {evolution_o}
We have the following evolution equations along $\a(t)$:
\begin {equation*}
E_1<\n_{E_2} E_1, E_2>=-\e-<\n_{E_2} E_1, E_2>^2
\end{equation*}
and
\begin {equation*}
E_1<\n_{E_3}E_1, E_3>=-\e-<\n_{E_3} E_1, E_3>^2.
\end{equation*}

\end{lem}

\begin{proof}
 
To prove the first equality, first use the fact that $E_1$ is isocurved and the fact that $\n_{E_1} E_1=0$ to show
\begin{equation} \label{evolution_o_1}
\e=R(E_2, E_1, E_1, E_2)
=<-\n_{E_1} \n_{E_2} E_1-\n_{[E_2, E_1]} E_1, E_2>.
\end{equation}
By (\ref{parallel_fs}),
\begin{equation} \label{evolution_1}
<\n_{E_1} \n_{E_2} E_1, E_2>=E_1<\n_{E_2} E_1, E_2>.
\end{equation}

Use (\ref{parallel_fs}) and (\ref{vanish_diagonal}) to show 
\begin{equation} \label {evolution_2}
\begin{split}
&<\n_{[E_2, E_1]} E_1, E_2>\\
=&<\n_{\n_{E_2} E_1-\n_{E_1}E_2} E_1, E_2>\\
=&<\n_{\n_{E_2} E_1}E_1, E_2>\\
=&<\n_{<\n_{E_2}E_1,E_2>E_2+<\n_{E_2} E_1, E_3>E_3} E_1 , E_2>\\
=&<\n_{E_2} E_1, E_2><\n_{E_2} E_1, E_2>+<\n_{E_2} E_1, E_3><\n_{E_3} E_1, E_2>\\
=&<\n_{E_2} E_1, E_2>^2.
\end{split}
\end{equation}
The first equality is now obtained by substituting (\ref{evolution_1}) and (\ref{evolution_2}) into (\ref{evolution_o_1}).

An analogous proof shows the second inequality.

\end{proof}

\begin{dfn}
Define an operator $A: E_1^{\perp}\to E_1^{\perp}$ by
$A=\n E_1$.
\end{dfn}

As $E_1$ has unit length, $A$ is well-defined. In terms of the orthonormal frame $\{E_2, E_3\}$ for $E_1^{\perp}$,
\begin{equation*}
tr(A)=<\n_{E_2} E_1, E_2>+<\n_{E_3} E_1, E_3>.
\end{equation*} 

In Lemma \ref{evo_equal}, we show that 
$A$ is given by scalar multiplication along $\a$.

\begin{lem}\label{evo_equal}
Along the geodesic $\a(t)$, we have that
\begin{enumerate}

\item $<\n_{E_2} E_1, E_3>=<\n_{E_3} E_1, E_2>=0$

\item $<\n_{E_2} E_1, E_2>=<\n_{E_3}E_1, E_3>.$

\end{enumerate}  
\end{lem}

\begin {proof}
To prove (1), it suffices to prove $<\n_{E_3} E_1, E_2>=0$, by Lemma \ref{parallel_lem}. 
To prove that $<\n_{E_3} E_1, E_2>=0$, first use (\ref{X_is_geodesic_field}) and (\ref{parallel_fs}) to show that
\begin{equation} \label{evo_equal_2}
\begin{split} 
&R(E_1, E_3, E_2, E_1)\\&=E_1< \n_{E_3} E_2,E_1> - <\n_{[E_1, E_3]} E_2, E_1>\\
&=E_1<\n_{E_3} E_2, E_1>+<\n_{<\n_{E_3} E_1, E_2>E_2+<\n_{E_3}E_1 ,E_3>E_3}E_2, E_1>\\
&=E_1<\n_{E_3} E_2, E_1>+tr(A)<\n_{E_3}E_2, E_1>
\end{split}
\end{equation}
By (\ref{x}) and (\ref{evo_equal_2}),
\begin{equation*} 
E_1<\n_{E_3} E_2, E_1>=-tr(A)<\n_{E_3}E_2, E_1>.
\end{equation*}
Hence if $<\n_{E_3} E_2, E_1>=0$ at $\a(0)$, then $<\n_{E_3} E_2, E_1>=0$ along $\a$ by the uniqueness of solutions to ordinary differential equations.

Direct calculations show that
\begin{equation*}
<\n_{E_3} E_2, E_1>=\frac{e_3(m)}{1+m^2}-\G_{31}^2.
\end{equation*}
Hence, at points in $\mathcal{F}^+$,  the last equality in (\ref{point_p}) implies that $<\n_{E_3}E_2, E_1>=0$, concluding the proof of (1).

To prove (2), recall from Lemma \ref{evolution_o} that $<\n_{E_3}E_1, E_3>$ and $<\n_{E_2} E_1, E_2>$ satisfy the same first order differential
equation, hence it suffices to prove (2) at a point along $\a(t)$.
Direct calculations show that 
\begin{equation*}
<\n_{E_3}E_1, E_3>=\frac{1}{\sqrt{1+m^2}}\G_{31}^3-\frac{m}{\sqrt{1+m^2}}\G_{32}^3,
\end{equation*}
and
\begin {equation} \label{evo_equal_3}
\begin{split}
<\n_{E_2} E_1, E_2>=&\frac{1}{\sqrt{1+m^2}}\frac{1}{1+m^2}(-me_1(m)-e_2(m))\\
-&\frac{m}{1+m^2}\{\frac{m^2}{\sqrt{1+m^2}}\G_{12}^1+\frac{m}{\sqrt{1+m^2}}\G_{22}^1\}\\
+&\frac{1}{1+m^2}\{\frac{m}{\sqrt{1+m^2}}\G_{11}^2+\frac{1}{\sqrt{1+m^2}}\G_{21}^2\}.
\end{split}
\end{equation}

At points in $\mathcal{F}^+$, equality holds as can be seen after substituting the second and third equalities in $(\ref{point_p})$
into (\ref{evo_equal_3}).

\end{proof}

We are now ready to prove the evolution equations along the geodesic $\a$. 
\begin {prop} \label{evo_tr}
Along the geodesic $\a$,
$$E_1(tr(A))=-2\e-\frac{1}{2}(tr(A))^2$$
\begin {equation} \label{evo_off}
\begin{split} 
E_1(R(E_1, E_2, E_2, E_3))&=-\frac{3}{2} tr(A)R(E_1, E_2, E_2, E_3)\\
E_1(R(E_1, E_3, E_3, E_2))&=-\frac{3}{2} tr(A)R(E_1, E_3, E_3, E_2).
\end{split}
\end{equation}

\end{prop} 

\begin {proof}
By Lemma \ref{evolution_o} and Lemma \ref{evo_equal},

\begin {equation*}
\begin {split}
E_1(tr(A))&=E_1(<\n_{E_2} E_1, E_2>+<\n_{E_3} E_1, E_3>)\\
&=-2\e-<\n_{E_2} E_1, E_2>^2-<\n_{E_3} E_1, E_3>^2\\
&=-2\e-\frac{1}{2}(tr(A))^2.
\end{split}
\end{equation*}

The first equality in (\ref{evo_off}) is obtained by expanding the second Bianchi identity
\begin{equation*}
(\n_{E_3} R)(E_1,E_2,E_2,E_1)+ (\n_{E_1} R)(E_2,E_3,E_2,E_1)+(\n_{E_2} R)(E_3,E_1,E_2,E_1)=0,
\end{equation*}
and then making simplifications using the Christoffel symbols relations
$$\n_{E_1} E_1= \n_{E_1} E_2 = \n_{E_1} E_3=0$$
$$<\n_{E_2} E_1, E_3>=<\n_{E_3} E_1, E_2>=0$$
$$<\n_{E_2} E_1, E_2> = <\n_{E_3} E_1, E_3>$$
obtained in Lemma \ref{parallel_lem} and Lemma \ref{evo_equal}.

The second equality in (\ref{evo_off}) is similarly obtained starting from the second Bianchi identity

\begin{equation*}
(\n_{E_3} R)(E_1,E_2,E_3,E_1)+ (\n_{E_1} R)(E_2,E_3,E_3,E_1)+(\n_{E_2} R)(E_3,E_1,E_3,E_1)=0.
\end{equation*}

\end{proof}


\section {Evolution Equations Along Isocurved Geodesics in $\mathcal{E}$} 

In this section, we assume that $\mathcal{O}$ is empty and that $\mathcal{E}$ is nonempty. Let $\mathcal{C}$ be a connected component of 
$\mathcal{E}$. 
By Proposition \ref{extreme_structure}, on $\mathcal{C}$ or a double cover of $\mathcal{C}$ (still denoted by $\mathcal{C}$), there exists a complete isocurved geodesic field $E$. By the definition of $\mathcal{E}_+$ and  $\mathcal{E}_-$, $\mathcal{C}\subset \mathcal{E}_-$ or 
$\mathcal{C}\subset \mathcal{E}_+$.

As the rank assumption passes to covers, we may apply Propositions \ref{point_e} and \ref {point_e_2} to derive evolution equations along $E$ geodesics on $\mathcal{C}$. 
The evolution equations are then applied to prove rigidity results for three-manifolds of higher rank with extremal curvature. These results will be used in the proof of the main theorems in the last section.

\begin{dfn}
Let $E$ be the isocurved geodesic field on $\mathcal{C}$. Define $\tilde{A}: E^{\perp}\to E^{\perp}$ by $\tilde{A}=\n E$. 
\end{dfn}

The trace of $\tilde{A}$ is clearly well-defined and smooth on $\mathcal{C}$. 

\begin {prop} \label{evo_extreme}
On $\mathcal{C}$, we have the following evolution equation for $tr (\tilde{A})$ along the $E$ geodesics:
$$E(tr \tilde{A})=-2\e-\frac{1}{2}(tr\tilde{A})^2.$$
\end {prop}

\begin {proof}
We prove the case when $ \mathcal{C}\subset \mathcal{E}_-$. The proof for the case when $\mathcal{C} \subset \mathcal{E}_+$ is similar.
Let $\{e_1, e_2, e_3\}$ be a Ricci diagonalizing orthonormal frame on $\mathcal{E}_-$ satisfying (\ref{order}). To prove the proposition, first use Proposition \ref{point_e} to obtain
\begin{equation*}
R(e_1, e_3, e_3, e_1)=e_1(\G_{33}^1)-(\G_{33}^1)^2.
\end{equation*}
By Corollary \ref{isovector}, $e_1$ is isocurved, so that $R(e_1, e_3, e_3, e_1)=\e.$
Therefore,
\begin{equation} \label{evo_extreme_1}
e_1(\G_{31}^3)=-\e-(\G_{31}^3)^2.
\end{equation}

Note that on $\mathcal{E}_-$, we may assume that $E=e_1$. 
By Proposition \ref{point_e}, we have that on $\mathcal{E}_-$,
\begin{equation} \label{evo_extreme_2}
\G_{33}^1=\G_{22}^1.
\end{equation}
Hence by (\ref{evo_extreme_1}) and (\ref{evo_extreme_2}),
\begin{equation*}
\begin{split}
E (tr \tilde{A})=e_1(tr\tilde{A})&=e_1(\G_{31}^3+\G_{21}^2)\\
&=-2\e-\frac{1}{2}(tr\tilde{A})^2
\end{split}
\end{equation*}

\end{proof}

The following proposition is a special case for the characterization of three-manifolds of higher spherical rank. 
\begin{prop} \label{ext}
Let $M$ be a complete Riemannian three-manifold of higher spherical rank and extremal sectional curvature 1, then
M has constant sectional curvatures.
\end{prop}
\begin {proof}
Let $M$ be a three manifold of higher spherical rank and extremal curvature 1. 
Then $M=\mathcal{E}\cup \mathcal{I}$. To prove the rigidity result, it suffices to show that $\mathcal{E}$ is empty.

Suppose that $\mathcal{E}$ is nonempty. Let $\a:(-\infty, \infty)\to \mathcal{C}$ be a complete isocurved geodesic obtained by
integrating the complete geodesic field $E$ on $\mathcal{C}$. 
By Proposition \ref{evo_extreme},
\begin{equation} \label{ext_1}
\frac{d}{dt} (tr \tilde{A})=-2-\frac{1}{2}(tr\tilde{A})^2
\end{equation}
along $\a(t)$. 
The solution to the differential equation (\ref{ext_1}) has a finite time singularity. This contradicts with the fact that $tr \tilde{A}$ is a smooth function along $\a(t)$.
Hence $\mathcal{E}$ is empty, and $M$ has constant sectional curvature 1.

\end{proof}

The evolution equations also lead to a proof of Theorem \ref{cvc}. 

\begin{proof}
Let $M$ be a complete Riemannian three-manifold of higher hyperbolic rank with extremal curvature and finite volume. 

Suppose that $\mathcal{E}$ is nonempty. Let $\mathcal{C}$ be a connected component of $\mathcal{E}$. Then $\mathcal{C}$ and its double cover have finite volume. 
By Proposition \ref{evo_extreme}, we have the following evolution equation along each isocurved $E$ geodesic:
\begin{equation} \label{ext_h_1}
\frac{d}{dt}(tr \tilde{A})=2-\frac{1}{2}(tr\tilde{A})^2.
\end{equation}
 The solution of the differential equation (\ref{ext_h_1}) has a finite forward time singularity when $tr(\tilde{A})(0)<-2$, and it has a finite backward time
 singularity when $tr(\tilde{A})(0)>2$.
As $tr(\tilde{A}$) is a smooth function, $tr(\tilde{A})(0)\in[-2, 2]$ along each isocurved $E$ geodesic. 
Hence there exists an open set $U$ such that either $-2<tr(\tilde{A})(q)\leq 2$ for all $q\in U$ or $tr(\tilde{A})(q)=-2$ for all $q\in U$.

Let $\phi_t$ denote the flow generated by the complete geodesic field $E$. When $-2<tr(\tilde{A})(q)\leq 2$ for all $q\in U$, (\ref{ext_h_1}) implies that 
$\mathrm{div}(E)=tr(\tilde{A})\to 2$ on $\phi_t(U)$ as $t\to \infty$. As 
$$\mathrm{div}(E)=\frac{d}{dt} \int_{\phi_t(U)} d\mathrm{vol},$$
the volume of $\phi_t(U)$ is arbitrarily large as $t\to \infty$, contradicting the assumption that $\mathcal{C}$ has finite volume.

When $tr(\tilde{A})(q)=-2$ for all $q\in U$. An analogous argument shows that the volume for $\phi_{-t}(U)$ is arbitrary large as $t\to \infty$, contradicting the assumption that 
$\mathcal{C}$ has finite volume. Hence $\mathcal{E}$ is empty, and $M$ is hyperbolic.

\end{proof}

\section {Rigidity Theorems}
We prove the main theorems in this section. In the proof of each theorem, we first 
exclude the possibility that $\mathcal{O}$ is nonempty. Then we use Proposition \ref{ext} or Theorem \ref{cvc} 
to obtain the desired rigidity result. Let $\mathcal{C}$ be a connected component of $\mathcal{O}$ as in Remark \ref{uniqueness}.

\begin{dfn}
On the connected component $\mathcal{C}\subset \mathcal{O}$, define functions $R_2$ and $R_3$ by
$$R_2=R(E_1, E_2, E_2, E_3)$$
and
$$R_3=R(E_1, E_3, E_3, E_2).$$
\end{dfn}

Along a geodesic $\a(t)$ constructed in Proposition \ref{tube}, we may solve the differential equations (\ref{evo_off}) to obtain
\begin{equation} \label {sph_rigid_4}
|R_i(t)|=|R_i(0)|\exp {\Big(-\frac{3}{2}\int_{0}^t tr(A)(s) ds} \Big)
\end{equation}
for $i=2,3$.

Recall from Lemma \ref{x_lemma} that either $R_2\neq0$ or $R_3\neq 0$ at each point in $\mathcal{C}$. 
The following lemma is a useful observation for the proof of both theorems.
\begin{lem} \label{off_zero}
Let $p\in \mathcal{C}$ and let $\a: (k, l) \rightarrow  \mathcal{C}$ be a maximal integral curve of $E_1$ with $\a(0)=p$. Furthermore, assume that
$\a(t)$ is a geodesic.
\begin{enumerate}
\item If $l$ is finite, then $\lim_{t\to l^-}R_i(\a(t))=0$, for $i=2,3$. 
\item If $k$ is finite, then $\lim_{t\to k^+} R_i(\a(t))=0$, for $i=2,3$.
\end{enumerate}
\end{lem}

\begin{proof}
We prove (1). The proof of (2) is analogous. Suppose that $l<\infty$, then $\a(l)\in \mathcal{E}\cup \mathcal{I}$. 

As $M$ is complete, we may extend $\a: (k, l) \rightarrow  \mathcal{C}$ to a complete geodesic $\a: \mathbb{R}\to M$. Let $\{E_1(t), E_2(t), E_3(t)\}$ be the orthonormal framing  determined by parallel translating the orthonormal frame $\{E_1, E_2, E_3\}$ at $\a(0)$
along $\a(t)$. By (\ref{parallel_fs}), $\{E_1(t), E_2(t), E_3(t)\}$ agrees with the orthonormal framing $\{E_1, E_2, E_3\}$ of $\mathcal{C}$ as long as $\a(t)\in \mathcal{C}$.

If $\a(l)\in \mathcal{I}$, then the formula for a curvature tensor at an isotropic point implies that
$$R(E_1(l), E_2(l), E_2(l), E_3(l))=0=R(E_1(l), E_3(l), E_3(l), E_2(l)).$$ The lemma now follows from continuity of parallel translation and the curvature tensor.

Next, consider the case when $\a(l)\in \mathcal{E}$. Then exactly one of $\l(t)$ or $\L(t)$ converges to $\e$ as $t\to l^-$.
If $\l(t)\to \e$, then $m\to \infty$. The formulae for $E_1$, $E_2$ and $E_3$ in $\mathcal{C}$ then imply that the distances
between the ordered orthonormal frames $\{E_1(t), E_2(t), E_3(t)\}$ and $\{-e_2(t), e_1(t), e_3(t)\}$ converge to zero as $t\to l^-$.
Continuity of the curvature tensor and (\ref{Ricci_diag}) imply (1) in this case. Likewise, if $\L \to \e$, then $m\to 0$. In this case, 
the distances between the ordered orthonormal frames $\{E_1(t), E_2(t), E_3(t)\}$ and $\{e_1(t), e_2(t), e_3(t)\}$ converge to zero as 
$t\to l^+$, concluding the proof similarly. 

\end{proof}

\subsection{Spherical Rank Rigidity}

For three-manifolds of higher spherical rank, we have the following theorem:
\begin{thm} 
A complete Riemannian three-manifold $M$ has higher spherical rank if and only if $M$ has constant sectional curvatures one.  
\end{thm}

\begin {proof}
By Proposition \ref{ext}, it suffices to show that $\mathcal{O}$ is empty.
Suppose that $\mathcal{O}$ is nonempty. By Proposition \ref{tube}, there exists a geodesic $\a$ such that $\a(0)\in \mathcal{F}^+$ and $\a'=E_1$ on a connected component $\mathcal{C}$ of $\mathcal{O}$. 
Let $(k, l)$ be the maximal open interval such that 
$\a((k, l)) \subset \mathcal{C}$.

Set $\e=1$ in Proposition \ref{evo_tr}
to obtain 
\begin{equation} \label{sph_rigid_2}
\frac{d}{dt}(tr(A))=-2-\frac{1}{2}(tr(A))^2.
\end{equation}

If $l<\infty$, we obtain a contradiction as follows. 
By Lemma \ref{x_lemma}, either $|R_2(0)|>0$ or $|R_3(0)|>0$. Assume that $|R_2(0)|>0$. The argument for
the case when $|R_3(0)|>0$ is similar.
Note that $-tr(A)$ is well-defined and smooth on $(k, l)$. Furthermore, (\ref{sph_rigid_2}) implies that $-tr(A)$ is strictly increasing on $(k, l)$.
Hence by (\ref{sph_rigid_4}), 
\begin{equation*}
\begin{split}
\lim_{t\to l^-} |R_2(t)|&=|R_2(0)|\exp {\Big(-\frac{3}{2}\int_{0}^l tr(A)(s) ds} \Big)\\
&>|R_2(0)|\exp {\Big(-\frac{3}{2} tr(A)(0)\cdot l } \Big)>0.
\end{split}
\end{equation*}
This is a contradiction to Lemma \ref{off_zero}. 

In the case when $l=\infty$, a contradiction is obtained as the solution to (\ref{sph_rigid_2}) has a finite forward time singularity.

\end{proof}

\subsection{Hyperbolic Rank Rigidity}
We first prove the following property for the isocurved geodesics constructed in Proposition \ref{tube}. 
\begin{lem} \label{hyp_lemma}
Let $\a(t)$ be an isocurved geodesic as constructed in Proposition \ref{tube} and let $(k,l)$ be the maximal open interval containing zero 
such that $\a((k,l))\subset \mathcal{C}$. 
\begin{enumerate}
\item $l=\infty$ if $tr(A)(0)\geq-2$.
\item $k=-\infty$ if $tr(A)(0)\leq -2$.
\end{enumerate}
\end{lem}

\begin{proof}
 
By taking $\e=-1$ in Proposition \ref{evo_tr}, we have the following evolution equation along $\a(t)$ on $(k,l)$:

\begin{equation}\label {trace_hyper}
\frac{d}{dt}(tr(A))=\frac{1}{2}(2-tr(A))(2+tr(A)).
\end{equation}

Note that by Lemma \ref{x_lemma}, either $|R_2(0)|>0$ or $|R_3(0)|>0$. Assume that $|R_2(0)|>0$. The argument for
the case when $|R_3(0)|>0$ is similar.

To prove (1), suppose that $l<\infty$ when $tr(A)(0)\geq-2$.  A straightforward analysis on the stability of equilibrium points for (\ref{trace_hyper}) shows that $tr(A)(t)$ is bounded on $[0,l]$ when $tr(A)(0)\geq-2$. Hence by (\ref{sph_rigid_4}), we have 
\begin{equation*}
\lim_{t\to l^-} |R_2(t)|=|R_2(0)|\exp {\Big(-\frac{3}{2}\int_{0}^l tr(A)(s) ds}\Big) >0.
\end{equation*}
This is a contradiction to Lemma \ref{off_zero}, concluding the proof of (1).

To prove (2), suppose that $k$ is finite when $tr(A)(0)\leq -2$. 
In this case, a straightforward analysis on the stability of equilibrium points for (\ref{trace_hyper}) shows that $tr(A)(t)$ is bounded on $[k, 0]$. By (\ref{sph_rigid_4}),
\begin{equation}
\lim_{t\to k^+} |R_2(t)|=|R_2(0)|\exp {\Big(-\frac{3}{2}\int_{0}^k tr(A)(s) ds}\Big) >0,
\end{equation}
contradicting Lemma \ref{off_zero}. Hence $k=-\infty$, concluding the proof of (2).

\end{proof}

\begin {thm}
A complete finite volume Riemannian three-manifold $M$ has higher hyperbolic rank if and only if $M$ is hyperbolic.
\end{thm} 

\begin{proof}
By Theorem \ref{cvc}, it suffices to show that $\mathcal{O}$ is empty. Suppose that $\mathcal{O}$ is nonempty. Let $U$ be the open subset of $\mathcal{C}$ as in Proposition \ref{tube}.
For each $p\in U$, let $\varphi_t (p)$ be the integral curve of $E_1$ starting from $p$. 

Note that after possibly shrinking the open set $U$, we may assume that either
$tr(A)(p)> -2$ for all $p\in U$ or $tr(A)(p)\leq -2$ for all $p\in U$. 

If $tr(A)(p)> -2$ for all $p\in U$, then Lemma \ref{hyp_lemma} implies that $\varphi_t(U)\subset \mathcal{C}$ for all $t>0$. Furthermore, (\ref{trace_hyper})
implies that $tr(A)$ approaches to $2$ on $\varphi_t(U)$ as $t$ approaches to infinity. As
\begin{equation}
tr(A)=\mathrm{div}(E_1)=\frac{d}{dt} \int_{\varphi_t (U)} d\mathrm{vol},
\end{equation}
the volume of $\mathcal{C}$ is infinite, a contradiction to the 
assumption that $M$ has finite volume. Hence $\mathcal{O}$ is empty.

Likewise, if $tr(A)(p)\leq -2$ for all $p\in U$, Lemma \ref{hyp_lemma} implies that $\varphi_{-t} (U)\subset \mathcal{C}$ for all $t>0$. The equation
(\ref{trace_hyper}) again shows that 
$$\mathrm{div}(-E_1)=-\mathrm{div}(E_1)\to 2$$ 
on $\varphi_{-t} (U)$ as $t$ approaches to infinity. Hence an analogous argument shows that 
$\mathcal{C}$ has infinite volume, contradicting to the assumption that $M$ has finite volume. Hence $\mathcal{O}$ is empty. This concludes the proof of the theorem.

\end{proof}

We now prove Theorem \ref{thm_3}.
\begin{thm}
A homogeneous three-manifold $M$ of higher hyperbolic rank is hyperbolic.
\end{thm}

\begin{proof}
By Theorem \ref{cvc_hom}, it suffices to rule out the case when $M=\mathcal{O}$. 
Suppose that $M=\mathcal{O}$. 
Since $M$ is homogeneous, $R_2$ and $R_3$ are constant on $\mathcal{C}$. At each point in $\mathcal{C}$, either $R_2\neq 0$ or $R_3\neq0$.
By (\ref{evo_off}), $tr(A)=0$ on $\mathcal{C}$, contradicting with (\ref{evo_tr}). 

\end{proof}

\bibliographystyle{plain}
\bibliography{Rank}


\end{document}